\documentclass[10pt]{article}
	\usepackage{latexsym}
	\usepackage{amssymb}
	\usepackage{graphicx,epsfig}
	\usepackage{graphicx}
	\usepackage{amsfonts}
	\usepackage{amsmath}
	\usepackage{amsmath,amsthm}
	
	\makeatletter
	\newcommand{\singlespacing}{\let\CS=\@currsize\renewcommand{\baselinestretch}{1}\tiny\CS}
		\newcommand{\oneandahalfspacing}{\let\CS=\@currsize\renewcommand{\baselinestretch}{1.25}\tiny\CS}
		\newcommand{\doublespacing}{\let\CS=\@currsize\renewcommand{\baselinestretch}{1.35}\tiny\CS}

	\def\@citex[#1]#2{\if@filesw\immediate\write\@auxout{\string\citation{#2}}\fi
			\def\@citea{}\@cite{\@for\@citeb:=#2\do
			{\@citea\def\@citea{,\linebreak[0]\hskip0pt plus .2em}%
				\@ifundefined{b@\@citeb}%
				{{\bf ?}\@warning{Citation `\@citeb' on page \thepage\space undefined}}%
				\hbox{\csname b@\@citeb\endcsname}}}{#1}}
	
	\newtheorem{Theorem}{Theorem}
	
	\newtheorem{Lemma}{Lemma}
	\date{}
	
\begin{document}
		
		
		\title{\bf JACOBI SUMS OF ORDER $2l^{2}$ }

\author{Md. Helal Ahmed and Jagmohan Tanti}
		\date{}	\maketitle \setlength{\parskip}{.11in}
		\setlength{\baselineskip}{15pt}
		\begin{center}\textbf{\underline{Abstract}}\end{center}
		Let $l\geq3$ and $p$ be primes, $q=p^r$, $r\in\mathbb{Z}^{+}$, $q\equiv1\pmod{2l^2}$ and $\mathbb{F}_q$ a field with $q$ elements. In this paper we establish the 
		congruences for Jacobi sums of order $2l^2$ and  also explore here to express these  
		Jacobi sums in terms of Dickson-Hurwitz sums. These expressions and congruences are useful in algebraic characterizations of the Jacobi 
		sums of order $2l^2$. 


\section{\underline{\textit{Introduction}}}
Let $e\geq2$ be an integer, $p$ a rational prime,  $q=p^{r}, r \in\mathbb{Z}^{+}$ and $q\equiv 1 \pmod{e}$. Let $\mathbb{F}_{q}$ 
be a finite field of $q$ elements. We can write $q=p^{r}=ek+1$ for some $k\in\mathbb{Z}^{+}$. Let $\gamma$ be a generator of 
the cyclic group $\mathbb{F}^{*}_{q}$ and $\zeta_e=exp(2\pi i/e)$. 
Define a multiplicative character $\chi_e : \ \mathbb{F}^{*}_{q} \longrightarrow \mathbb{Q}(\zeta_e)$ by $\chi_e(\gamma)=\zeta_e$ and extend it on 
$\mathbb{F}_q$ 
by putting $\chi_e(0)=0$. 
For  integers $\displaystyle 0\leq i,j\leq e-1$, the Jacobi sum $J_{e}(i,j)$ is define by

$$J_{e}(i,j)= \sum_{v\in \mathbb{F}_{q}} \chi_e^{i}(v) \chi_e^{j}(v+1).$$

However in the literature a variation of Jacobi sums are also considered and is defined by 
$$J_e(\chi_e^i,\chi_e^j)=\sum_{v\in\mathbb{F}_q}\chi_e^i(v)\chi_e^j(1-v),$$ but are related by $J_e(i,j)=\chi_e^i(-1)J_e(\chi_e^i,\chi_e^j)$.

The Problem of congruences of Jacobi sums of order $e$ concerns to determine an element modulo an appropriate power of  
 $(1-\zeta_e)$ in $\mathbb{Z}[\zeta_e]$, which is coprime to $e$ and this plays a vital role to determine the element uniquely along with 
some other elementary conditions. 

For some small values of $e$ the study of congruences of Jacobi sums is available in the literature. For example
Dickson \cite{Dickson} obtained the congruences $J_{l}(1,n)\equiv -1 \pmod{(1-\zeta_{l})^{2}}$ for $1\leq n \leq l-1$. Parnami, Agrawal and Rajwade 
\cite{Parnami} also calculated this separately. Iwasawa \cite{Iwasawa} in $1975$, and in $1981$ Parnami, Agrawal and Rajwade \cite{Parnami 2} 
established that the above congruences also hold $\pmod{(1-\zeta_{l})^{3}}$. Further in $1995$, Acharya and Katre \cite{Katre} extended the work on finding
the congruences for 
Jacobi sums and  showed that 
\begin{center}
$J_{2l}(1,n)\equiv -\zeta_l^{m(n+1)}($mod$\ (1-\zeta_{l})^{2})$,
\end{center} 
where $n$ is an odd integer such that $1\leq n \leq 2l-3$ and $m=$ind$_{\gamma}2$.
Also in $1983$, Katre and Rajwade \cite{Rajwade} obtained the congruence of Jacobi sum of order $9$, i.e.,
\begin{center}
$J_{9}(1,1)\equiv -1-($ind$\ 3)(1-\omega)($mod$\ (1-\zeta_{9})^{4})$,
\end{center} 
where $\omega = \zeta_{9}^{3}$.
In $1986$, Ihara \cite{Ihara} showed that if $k>3$ is an odd prime power, then
 \begin{center}
 $J_{k}(i,j)\equiv -1 \pmod{(1-\zeta_{k})^{3}}$.
 \end{center}
 Evans (\cite{Evans}, $1998$) used simple methods to generalize this result for all $k>2$.
Congruences for the Jacobi sums of order $l^{2}$ ($l$ odd prime) were obtained by Shirolkar and Katre \cite{lsquare}. They showed that\\ 
$J_{l^{2}}(1,n)\equiv
 \begin{cases}
  -1 + \sum_{i=3}^{l}c_{i,n} (\zeta_{l^2} -1)^{i} ($mod$\ (1-\zeta_{l^2})^{l+1}) \ \ \ \ \ if\ $gcd$ (l,n)=1, \\
  -1 \ ($mod$\ (1-\zeta_{l^2})^{l+1}) \ \ \ \ \ \ \  \ \ \ \ \ \ \ \ \ \ \ \ \ \ \ \ \ \ \ \ \ \ if\ $gcd$ (l,n)=l.  
 \end{cases}$ \\ \\ 
In this paper, we determine the congruences $\pmod{(1-\zeta_{2l^2})^{l+1}}$ for Jacobi sums of order $2l^{2}$  and also calculate their coefficients 
in terms of Dickson-Hurwitz sums.  We split the problem into two cases:\\ 
\textbf{Case 1}. $n$ is odd. This case splits into three sub-cases:\\
 \textbf {Subcase i}. $n=l^{2}$.\\
 \textbf {Subcase ii}. $n=dl, \ $where$\ 1\leq d\leq 2l-1, d$ is an odd and $d\neq l$.\\
 \textbf {Subcase iii}. gcd$(n,2l^{2})=1$.\\ \\
 \textbf{Case 2}. $n$ is even. In this case the Jacobi sums $J_{2l^{2}}(1,n)$ can be calculated using the relation 
 $J_{2l^{2}}(1,n)=\chi_{2l^2}(-1)J_{2l^{2}}(1,2l^{2}-n-1)$ (which has been shown in the next section).\\ \\  
This paper has been organized in the following pattern. Section 2, discussess some properties of Jacobi sums, Dickson-Hurwitz sums and relations 
among them. These properties are used to evaluate the congruences of Jacobi sums and their expressions in terms of Dickson-Hurwitz sums. 
In Section 3, we prove some Lemma's, which are needed for the proofs of congruences for Jacobi sums. Section 4 is assigned to discuss the main results, 
i.e. to determine Jacobi sums of order $2l^{2}$ in terms of the coefficients of Dickson-Hurwitz sums and to calculate the 
congruences for Jacobi sums of order $2l^{2}$.  
\section{Preliminaries}
Let $\zeta=\zeta_{2l^2}$ and $\chi=\chi_{2l^2}$ then $\chi^2=\chi_{l^2}$ is a character of order $l^2$ and $\zeta_{l^2}=\zeta^2$ is a primitive $l^2$th 
root of unity. The Jacobi sums $J_{l^{2}}(i,j)$ and $J_{2l^{2}}(i,j)$ of order $l^{2}$ and $2l^{2}$ respectively are defined as in the previous section. 
We also have $\zeta=-\zeta_{l^2}^{(l^{2}+1)/2}$.
\subsection{Properties of Jacobi sums} 
In this subsection we discus some properties of Jacobi sums.
\\ \\
\textbf{Proposition 1}.
If $m+n+s\equiv 0 \ ($mod$\ e)$ then 
\begin{equation*}
J_{e}(m,n)=J_{e}(s,n)=\chi_e^{s}(-1)J_{e}(s,m)=\chi_e^{s}(-1)J_{e}(n,m)=\chi_e^{m}(-1)J_{e}(m,s)=\chi_e^{m}(-1)J_{e}(n,s).
\end{equation*}  	
In particular,
\begin{equation*}
J_{e}(1,m)=\chi_e(-1)J_{e}(1,s)=\chi_e(-1)J_{e}(1,e-m-1).
\end{equation*} 
\textbf{Proposition 2}. 
\ \ \ \ $J_{e}(0,j)= \begin{cases}
-1 \ \ \ \ \ if  \ j\not\equiv 0 \ ($mod$\ e) ,\\
q-2 \ \ if  \ j\equiv 0 \ ($mod$\ e).
\end{cases}$

\ \ \ \ $J_{e}(i,0)=-\chi_e^{i}(-1) \ if \  i\not\equiv 0 \ ($mod$\ e).$
\\ \\ 
\textbf{Proposition 3}.  \ \ \ \ Let $m+n\equiv 0 \ ($mod$\ e)$ \ but not both $m$ and $n$ zero $\pmod{e}$. Then $J_{e}(m,n)=-1.$
\\ \\ \textbf{Proposition 4}.  \ \  For $(k,e)=1$ and $\sigma_k$ a $\mathbb{Q}$ automorphism of $\mathbb{Q}(\zeta_e)$ with $\sigma_k(\zeta_e)=\zeta_e^k$, 
we have $\sigma_{k}J_{e}(m,n)=J_{e}(mk,nk)$. In particular, if $(m,e)=1,\ m^{-1}$ \ denotes the inverse of $m\pmod{e}$ then 
$\sigma_{m^{-1}}J_{e}(m,n)=J_{e}(1,nm^{-1})$.
\\ \\ \textbf{Proposition 5}.  \ \ \ \ 
$J_{2e}(2m,2s)=J_{e}(m,n)$. \\ \\
\textbf{Proposition 6}. \ \ \ Let $m$, $n$, $s$ be integers such that $m+n \not\equiv 0 \ ($mod$\ 2l^{2})$ and $m+s \not\equiv 0 \ ($mod$\ 2l^{2})$. Then 
\begin{equation*}
J_{2l^{2}}(m,n) J_{2l^{2}}(m+n ,s) = \chi^{m}(-1)J_{2l^{2}}(m,s)J_{2l^{2}}(n,s+m).
\end{equation*}
\\  \textbf{Proposition 7}. \ \ \ \ $J_{2l^{2}}(1,n) \overline{J_{2l^{2}}(1,n)}= \begin{cases}
q \ \ \ \ \ if \ n\not\equiv 0, -1 \ ($mod$\ 2l^{2}),\\
1 \ \ \ \ \ if \ n\equiv 0, -1 \ ($mod$\ 2l^{2}).
\end{cases}$
\begin{proof}
The proofs of 1 - 5 follows directly using the definition of Jacobi sums (see \cite{Jacobi}). The proofs of 6 and 7 are analogous to the proofs in the 
$2l$ and $l$ cases respectively (see \cite{Katre}, \cite{Parnami}). 
\end{proof}
\noindent\textbf{Remark 1:} The Jacobi sums of order $2l^{2}$ can be determined from the Jacobi sums of order $l^{2}$. The Jacobi sums of order $2l^2$
can also be obtained from $J_{2l^{2}}(1,n)$, $1\leq n\leq 2l^{2}-3$ for $n$ odd (or equivalently, $2\leq n\leq 2l^{2}-2$ for $n$ even). Further the Jacobi 
sums of  
order $l^{2}$ can be evaluated if one knows the Jacobi sums $J_{l^{2}}(1,i)$, $1\leq i\leq \frac{l^{2}-3}{2}$.

\subsection{Dickson-Hurwitz sums}
The Dickson-Hurwitz sums \cite{lsquare} of order $e$ on $\mathbb{F}_{q}$ are defined for $i,j \pmod{e}$ by
\begin{equation*}
B_{e}(i,j)=\sum_{h=0}^{e-1}(h,i-jh)_e.
\end{equation*}
For $q=ek+1$ these satisfy the relations:
\\ $B_{e}(i,j)=B_{e}(i,e-i-j),$ \ \ $B_{e}(0,0)=k-1$, \ \ $B_{e}(i,0)=k$ \ if $1\leqslant i \leqslant e-1$, \ \ $\sum_{i=0}^{e-1}B_{e}(i,j)=q-2$, and 
for $(j,e)=1$, \ \ $B_{e}(i\overline{j},\overline{j})=B_{e}(i,j)$, \ \ where $k\overline{k}\equiv 1 \pmod{e}$. \\  \\
Jacobi sums $J_{e}(\chi_e,\chi_e^{j})$ and Dickson-Hurwitz sums are related by (for q=p, see \cite{Dickson})
\begin{equation*}
\chi_e^{j}(-1)J_{e}(\chi_e,\chi_e^{j})=\chi_e^{j}(-1)\chi_e(-1)J_{e}(1,j)=\sum_{i=0}^{e-1}B_{e}(i,j)\zeta^{i}_{e}.
\label{1.1} \end{equation*}
Thus if k is even or $q=2^{r}$ then $J_{e}(1,j)=\sum_{i=0}^{e-1}B_{e}(i,j)\zeta^{i}_{e}$ \cite{lsquare}. Further
Shirolkar and Katre \cite{lsquare} calculated the coefficients of Jacobi sums of order $l^{2}$ in terms of Dickson-Hurwitz sums, which has been stated in
the following proposition:
 \\ \\
 \textbf{Proposition 8.} \  \ \  \  $J_{l^{2}}(1,n)=\sum_{i=0}^{l(l-1)-1} b_{i,n}\zeta_{l^2}^{i}$,
 \\where $b_{i,n}=B_{l^{2}}(i,n)-B_{l^{2}}(l(l-1)+j,n),\, 0\leq j \leq l-1, \ \ j\equiv i \pmod{l}$.
 
\section{Congruences for Jacobi sums of order $2l^{2}$}
The evaluation of congruences for the Jacobi sums of order $l^{2}$ has been done by Shirolkar and Katre \cite{lsquare} and has been stated in the following
theorem.
\begin{Theorem} \label{T1}
Let $l>3$ be a prime and $p^{r}=q\equiv 1 \ (\emph{mod}\ l^{2})$. If $1\leq n \leq l^{2}-1$, then a (determining) congruence for $J_{l^{2}}(1,n)$ for a 
finite field $\mathbb{F}_{q}$ is given by \\ $J_{l^{2}}(1,n)\equiv
\begin{cases}
-1 + \sum_{i=3}^{l}c_{i,n} (\zeta_{l^2} -1)^{i} (\emph{mod}\ (1-\zeta_{l^2})^{l+1}) \ \ \ \ \ if\ \emph{gcd} (l,n)=1, \\
 -1 \ (\emph{mod}\ (1-\zeta_{l^2})^{l+1}) \ \ \ \ \ \ \  \ \ \ \ \ \ \ \ \ \ \ \ \ \ \ \ \ \ \ \  if\ \emph{gcd} (l,n)=l,  
\end{cases}$  \\ 
where for $3\leq i \leq l-1$, $c_{i,n}$ and $c_{l,n}=S(n)$ are as described in \cite{lsquare}.
\end{Theorem}
\begin{Lemma} \label{L1}
Let $p\geq 3$ be a prime and $q=p^{r}\equiv 1 \pmod{2l^{2}}$. If $\chi$ is a nontrivial character of order $2l^{2}$ on the finite field $\mathbb{F}_{q}$, 
then
\begin{equation*}
J_{2l^{2}}(a,a)=\chi^{-a}(4)J_{2l^{2}}(a,l^{2}).
\end{equation*}
\end{Lemma}
\begin{proof}
As for $\alpha\in\mathbb{F}_{q}$, the number of $\beta\in\mathbb{F}_{q}$ satisfying the equation $\beta(1+\beta)=\alpha$ is same as 
$1+\chi^{l^{2}}(1+4\alpha)$, we have 
\begin{eqnarray*}
J_{2l^{2}}(a,a)&=&\sum_{\beta\in\mathbb{F}_{q}}\chi^{a}(\beta(1+\beta))= \sum_{\alpha\in\mathbb{F}_{q}}\chi^{a}(\alpha)\{1+\chi^{l^{2}}(1+4\alpha)\} \\
&=& \chi^{-a}(4)\sum_{\alpha\in\mathbb{F}_{q}}\chi^{a}(4\alpha)\chi^{l^{2}}(1+4\alpha) 
= \chi^{-a}(4) J_{2l^{2}}(a,l^{2}).
\end{eqnarray*}
\end{proof}
\begin{Lemma} \label{L2}
Let $l\geq3$ be a prime, $q=p^{r}\equiv 1 \ (\emph{mod}\  2l^{2})$ and $\gamma$ a generator of $\mathbb{F}_q^*$ then 
\begin{equation*}
J_{2l^{2}}(1,l^{2})\equiv \zeta_{l^2}^{-w}(-1+\sum_{i=3}^{l} c_{i,n}(\zeta_{l^{2}}-1)^{i})(\emph{mod}\ (1-\zeta_{l^{2}})^{l+1}),
\end{equation*}
where $w=ind_\gamma2$.
\end{Lemma}
\begin{proof}
From Proposition 6, for $m=1$, $n=1$ and $s=l^{2}-1$, we have
\begin{equation} \label{5.1}
J_{2l^{2}}(1,1) J_{2l^{2}}(2 ,l^{2}-1) = \chi(-1)J_{2l^{2}}(1,l^{2}-1)J_{2l^{2}}(1,l^{2}).
\end{equation}
By Proposition 1, we obtain $\chi (-1) J_{2l^{2}}(1, l^{2}-1)= J_{2l^{2}}(1,l^{2})$.\\ 
Now equation (\ref{5.1}) becomes 
\begin{equation} \label{5.2}
J_{2l^{2}}(1,1) J_{2l^{2}}(2, l^{2}-1)= J_{2l^{2}}(1,l^{2})J_{2l^{2}}(1,l^{2}).
\end{equation} 
Again by Proposition 5 and Theorem \ref{T1}, we have
\begin{equation} \label{5.3}
J_{2l^{2}}(2, l^{2}-1)=J_{l^{2}}(1, (l^{2}-1)/2)\equiv -1 + \sum_{i=3}^{l}c_{i,n} (\zeta_{l^{2}} -1)^{i} (\textup{mod}\ (1-\zeta_{l^{2}})^{l+1}). 
\end{equation}
For $w=ind_{\gamma}2$, from \textbf{Lemma \ref{L1}}, we have
\begin{equation} \label{5.4}
J_{2l^{2}}(1,1)=\chi^{-1}(4)J_{2l^{2}}(1,l^{2})=\zeta_{l^{2}}^{-w}J_{2l^{2}}(1,l^{2}).
\end{equation} 
Employing (\ref{5.4}) and (\ref{5.3}) in (\ref{5.2}), we get
\begin{equation*}
J_{2l^{2}}(1,l^{2})\equiv \zeta_{l^{2}}^{-w}(-1+\sum_{i=3}^{l} c_{i,n}(\zeta_{l^{2}}-1)^{i})(\textup{mod}\ (1-\zeta_{l^{2}})^{l+1}).
\end{equation*}
\end{proof}

\begin{Lemma} \label{L3}
Let $n$ be an odd integer such that $1\leq n \leq 2l^{2}-1$ and $\emph{gcd}(n,2l^{2})=1$, then
\begin{align*}
J_{2l^{2}}(1,n)&\equiv \zeta_{l^{2}}^{-w(n+1)}(-1+\sum_{i=3}^{l} c_{i,n}(\zeta_{l^{2}}-1)^{i})(-1+\sum_{i=3}^{l} c_{i,n}(\zeta_{l^{2}}^{n}-1)^{i})\\ &\times(-1+\sum_{i=3}^{l} c_{i,n}
(\zeta_{l^{2}}^{(1-l^{2})/2}-1)^{i})\pmod {(1-\zeta_{l^{2}})^{l+1}}.
\end{align*}
\end{Lemma}
\begin{proof}
By Proposition 6, for $m=1$ and $s=l^{2}-1$, we have
\begin{equation*}
J_{2l^{2}}(1,n) J_{2l^{2}}(1+n ,l^{2}-1) = \chi(-1)J_{2l^{2}}(1,l^{2}-1)J_{2l^{2}}(n,l^{2}).
\end{equation*}
Applying Proposition 1 to get
\begin{equation*}
J_{2l^{2}}(1,n) J_{2l^{2}}(1+n ,l^{2}-1) = J_{2l^{2}}(1,l^{2})J_{2l^{2}}(n,l^{2}).
\end{equation*}
Applying $\tau_{n}:\zeta_{2l^{2}}\rightarrow\zeta_{2l^{2}}^{n}$ (a $\mathbb{Q}$ automorphism of $\mathbb{Q}(\zeta_{2l^2})$) second term in RHS, we get
\begin{equation*}
J_{2l^{2}}(1,n) J_{2l^{2}}(1+n ,l^{2}-1) =J_{2l^{2}}(1,l^{2})\tau_{n}J_{2l^{2}}(1,l^{2}).
\end{equation*}
Now as
$J_{2l^{2}}(1+n,l^{2}-1) \overline{J_{2l^{2}}(1+n,l^{2}-1)}=q,$ we have
\begin{equation*}
J_{2l^{2}}(1,n)J_{2l^{2}}(1+n,l^{2}-1) \overline{J_{2l^{2}}(1+n,l^{2}-1)}= J_{2l^{2}}(1,l^{2})\tau_{n}J_{2l^{2}}(1,l^{2})
\overline{J_{2l^{2}}(1+n,l^{2}-1)}.
\end{equation*}
This implies
\begin{equation}  \label{6.5}
J_{2l^{2}}(1,n)\ q= J_{2l^{2}}(1,l^{2})\tau_{n}J_{2l^{2}}(1,l^{2})\overline{J_{2l^{2}}(1+n,l^{2}-1)}.
\end{equation}
From Proposition 5 and applying $\sigma_{n}:\zeta_{l^{2}}\rightarrow\zeta_{l^{2}}^{n}$ (a $\mathbb{Q}$ automorphism of $\mathbb{Q}(\zeta_{l^2})$), we get,
\begin{align*}
&\overline{J_{2l^{2}}(1+n,l^{2}-1)}=J_{2l^{2}}(-1-n,1-l^{2})=J_{l^{2}}((-1-n)/2,(1-l^{2})/2)\\ &=\sigma_{(1-l^{2})/2}(J_{l^{2}}(1,-1-n).
\end{align*}
Now from Theorem \ref{T1}, we get
\begin{equation} \label{6.9}
\overline{J_{2l^{2}}(1+n,l^{2}-1)}\equiv-1+\sum_{i=3}^{l} c_{i,n}(\zeta_{l^{2}}^{(1-l^{2})/2}-1)^{i}\pmod {(1-\zeta_{l^{2}})^{l+1}}.
\end{equation}
Employing (\ref{6.9}) and \textbf{Lemma \ref{L2}} in (\ref{6.5}), we get
\begin{align*}
J_{2l^{2}}(1,n) &\equiv \zeta_{l^{2}}^{-w(n+1)}(-1+\sum_{i=3}^{l} c_{i,n}(\zeta_{l^{2}}-1)^{i})(-1+\sum_{i=3}^{l} c_{i,n}(\zeta_{l^{2}}^{n}-1)^{i})\\ &\times(-1+\sum_{i=3}^{l} c_{i,n}
(\zeta_{l^{2}}^{(1-l^{2})/2}-1)^{i})\pmod {(1-\zeta_{l^{2}})^{l+1}}.
\end{align*}
\end{proof}

\begin{Lemma} \label{L5}
Let $d\neq l$, $1\leq d \leq 2l-1$ be an odd positive integer, $\gamma$ a generator of $\mathbb{F}_q^*$, then
\begin{align*}
& J_{2l^{2}}(1,dl)\equiv-\zeta_{l^{2}}^{-w(dl+1)}(-1+\sum_{i=3}^{l} c_{i,n}(\zeta_{l^{2}}-1)^{i})(-1+\sum_{i=3}^{l} c_{i,n}
(\zeta_{l^{2}}^{(-1-dl)/2}-1)^{i})(\textup{mod}\ (1-\zeta_{l^{2}})^{l+1}),
\end{align*}
where $w=ind_\gamma2$.
\end{Lemma}
\begin{proof}
From Proposition 6, for $m=1$, $n=dl$ and $s=l^{2}-1$, we have
\begin{equation*} \label{5.7}
J_{2l^{2}}(1,dl) J_{2l^{2}}(1+dl ,l^{2}-1) = \chi(-1)J_{2l^{2}}(1,l^{2}-1)J_{2l^{2}}(dl,l^{2}).
\end{equation*}
Applying Proposition 1, we get
\begin{equation*} \label{5.8}
J_{2l^{2}}(1,dl) J_{2l^{2}}(1+dl ,l^{2}-1) = J_{2l^{2}}(1,l^{2})J_{2l^{2}}(dl,l^{2}).
\end{equation*} 
Now as
$J_{2l^{2}}(1+dl,l^{2}-1) \overline{J_{2l^{2}}(1+dl,l^{2}-1)}=q,$
we have
\begin{equation*}
J_{2l^{2}}(1,dl) J_{2l^{2}}(1+dl ,l^{2}-1)\overline{J_{2l^{2}}(1+dl,l^{2}-1)} = J_{2l^{2}}(1,l^{2})J_{2l^{2}}(dl,l^{2})\overline{J_{2l^{2}}(1+dl,l^{2}-1)}.
\end{equation*}
This implies
\begin{equation}  \label{6.3}
J_{2l^{2}}(1,dl)\ q = J_{2l^{2}}(1,l^{2})J_{2l^{2}}(dl,l^{2})\overline{J_{2l^{2}}(1+dl,l^{2}-1)}.
\end{equation}
By Proposition 5, we get 
\begin{align*}
\overline{J_{2l^{2}}(1+dl,l^{2}-1)}=J_{2l^{2}}(-1-dl,1-l^{2})=J_{l^{2}}((-1-dl)/2,(1-l^{2})/2)=\sigma_{(-1-dl)/2}(J_{l^{2}}(1,dl-1).
\end{align*}
Now from Theorem \ref{T1}, we get
\begin{equation} \label{6.4}
\overline{J_{2l^{2}}(1+dl,l^{2}-1)}\equiv-1+\sum_{i=3}^{l} c_{i,n}(\zeta_{l^{2}}^{(-1-dl)/2}-1)^{i}(\textup{mod}\ (1-\zeta_{l^{2}})^{l+1})
\end{equation}
As $\chi$ is of order $2l^{2}$, $\chi^{l}$ is of order $2l$, so for $\eta_{d}:\zeta_{l}\rightarrow\zeta_{l}^{d}$ (a $\mathbb{Q}$ automorphism of $\mathbb{Q}(\zeta_l)$ and 
by Lemma $3$ \cite{Katre}, we obtain 
\begin{align*}
J_{2l^{2}}(dl,l^{2})&=J_{2l}(dl,l^{2})=(\chi^{l})^{d}(4)J_{2l}(d,d)=\chi^{ld}(2^{2})\eta_{d}(J_{2l}(1,1))=\chi^{2ld}(2)\eta_{d}(J_{2l}(1,1)).
\end{align*}
Let $w=ind_{\gamma}2$, then we have
\begin{align*}
\chi^{2ld}(2)=\chi^{2ld}(\gamma^{w})=\chi^{2ldw}(\gamma)=\zeta_{2l^{2}}^{2ldw}=\zeta_{l}^{dw}.
\end{align*}
By Proposition 3 \cite{Katre}, we get
\begin{align*}
J_{2l}(1,1)\equiv-\zeta_{l}^{-2w}(\textup{mod}\ (1-\zeta_{l})^{2}).
\end{align*}
So
\begin{align*}
\eta_{d}(J_{2l}(1,1))\equiv-\zeta_{l}^{-2wd}(\textup{mod}\ (1-\zeta_{l}^{d})^{2})\equiv-\zeta_{l}^{-2wd}(\textup{mod}\ (1-\zeta_{l})^{2}).
\end{align*}
Thus 
\begin{align*}
J_{2l^{2}}(dl,l^{2})&\equiv\zeta_{l}^{wd}(-\zeta_{l}^{-2wd})(\textup{mod}\ (1-\zeta_{l})^{2})\equiv-\zeta_{l}^{-wd}(\textup{mod}\ (1-\zeta_{l})^{2})\\ & 
\equiv-\zeta_{l^{2}}^{-wdl}(\textup{mod}\ (1-\zeta_{l^{2}}^{l})^{2})\equiv-\zeta_{l^{2}}^{-wdl}(\textup{mod}\ (1-\zeta_{l^{2}})^{2}).
\end{align*}
This implies 
\begin{equation}\label{5.9}
J_{2l^{2}}(dl,l^{2})(1-\zeta_{l^{2}})^{l-1}\equiv-\zeta_{l^{2}}^{-wdl}(1-\zeta_{l^{2}})^{l-1}(\textup{mod}\ (1-\zeta_{l^{2}})^{l+1}).
\end{equation}
Employing (\ref{5.9}), (\ref{6.4}) and Lemma \ref{L2} in (\ref{6.3}), we get
\begin{align*}
& J_{2l^{2}}(1,dl)\equiv-\zeta_{l^{2}}^{-w(dl+1)}(-1+\sum_{i=3}^{l} c_{i,n}(\zeta_{l^{2}}-1)^{i})(-1+\sum_{i=3}^{l} c_{i,n}
(\zeta_{l^{2}}^{(-1-dl)/2}-1)^{i})(\textup{mod}\ (1-\zeta_{l^{2}})^{l+1}).
\end{align*}
\end{proof}

\section{Main Theorems}
\begin{Theorem}
Let $p$ and $l\geq3$ be primes, $r\in\mathbb{Z}^{+}$, $q=p^{r}\equiv 1 (\emph{mod}\ 2l^{2})$ and $\zeta$ a primitive $2l^{2}$th root of unity. 
Then for $1\leq n\leq 2l^2-3$ the Jacobi sum $J_{2l^{2}}(1,n)$ of order $2l^{2}$ is given by \begin{equation*}
J_{2l^{2}}(1,n)=\sum_{i=0}^{l(l-1)-1} d_{i,n}\zeta^{i},
\end{equation*}
 with
 \begin{equation*}
d_{i,n}=B_{2l^{2}}(i,n)\mp B_{2l^{2}}(l(l-1)+j,n)-B_{2l^{2}}(l^{2}+k,n)\pm B_{2l^{2}}((2\phi(l^{2})+l)+j,n),
 \end{equation*}
 where
 \begin{equation*}
 0\leq j \leq l-1, \ j\equiv i \ (\emph{mod}\ 2l^{2}),\ k\equiv i \ (\emph{mod}\ 2l^{2}), \ 0\leq i,k \leq \phi(2l^{2})-1,
 \end{equation*}
 sign's are '$l$' periodic (i,e; sign's are repeated after $l$ terms).
\end{Theorem}
\begin{proof}
Cyclotomic polynomial of order $n=2l^{2}$ is
\begin{equation*}
\phi_{n}(x)=\sum_{i=0}^{l-1}(-x)^{il}=1-x^{l}+x^{2l}-x^{3l}+.......+x^{l(l-1)}
\end{equation*}
So, we have
\begin{equation*}
1-\zeta^{l}+\zeta^{2l}-\zeta^{3l}+.......+\zeta^{l(l-1)}=0
\end{equation*}
\begin{equation*}
\Rightarrow \zeta^{l(l-1)}=-1+\zeta^{l}-\zeta^{2l}+.......+\zeta^{l(l-2)}.
\end{equation*}
Every $\zeta^{i}, \  l(l-1)\leq i \leq 2l^{2}-1$ can be written as a linear combination of $\zeta^{i},0\leq i \leq l(l-1)-1$.
\begin{equation*}
J_{2l^{2}}(1,n)=\sum_{i=0}^{2l^{2}-1}B_{2l^{2}}(i,n)\zeta^{i}
\end{equation*}
\begin{align*}
&=B_{2l^{2}}(0,n)+B_{2l^{2}}(1,n)\zeta+B_{2l^{2}}(2,n)\zeta^{2}+......+B_{2l^{2}}(l(l-1)-1,n)\zeta^{l(l-1)-1} + B_{2l^{2}} (l(l-1),n)\\ & \ \ \ \  (-1+\zeta^{l}-\zeta^{2l}+.....+\zeta^{l(l-2)})+B_{2l^{2}}(l(l-1)+1,n)(-\zeta+\zeta^{l+1}-\zeta^{2l+1}+....+\zeta^{l(l-2)+1})+....\\& \ \ \  +B_{2l^{2}}(l(l-1)+l,n)(-\zeta^{l}+\zeta^{2l}-\zeta^{3l}+.....+\zeta^{l(l-2)+l})+.....+B_{2l^{2}}(2l(l-1)+l-1,n)\\& \ \ \ \ (-\zeta^{l(l-1)}-1)+B_{2l^{2}}(2l(l-1)+l,n)(-\zeta^{l(l-1)})+.....+B_{2l^{2}}(2l^{2}-1,n)(-\zeta^{l^{2}-1})
\end{align*}
\begin{align*}
&=B_{2l^{2}}(0,n)+B_{2l^{2}}(1,n)\zeta+B_{2l^{2}}(2,n)\zeta^{2}+......+B_{2l^{2}}(l(l-1)-1,n)\zeta^{l(l-1)-1}+B_{2l^{2}}(l(l-1),n)\\ & \ \ \ \ (-1+\zeta^{l}-\zeta^{2l}+.....+\zeta^{l(l-2)})+B_{2l^{2}}(l(l-1)+1,n)(-\zeta+\zeta^{l+1}-\zeta^{2l+1}+....+\zeta^{l(l-2)+1})+....\\ & \ \ \ +B_{2l^{2}}(l(l-1)+l,n)(-l)+.....+B_{2l^{2}}(2l(l-1)+l-1,n)(-\zeta^{l(l-1)}-1)+B_{2l^{2}}(2l(l-1)+l,n)\\ & \ \ \ \ (-\zeta^{l(l-1)})+.....+B_{2l^{2}}(2l^{2}-1,n)(-\zeta^{l^{2}-1})
\end{align*}

\begin{align*}
&=B_{2l^{2}}(0,n)-B_{2l^{2}}(l(l-1),n)-B_{2l^{2}}(l(l-1)+l,n)+B_{2l^{2}}(2l(l-1)+l,n)+\zeta(B_{2l^{2}}(1,n)\\ & \ \ \ -B_{2l^{2}}(l(l-1)+1,n)-B_{2l^{2}}(l(l-1)+l+1,n)+B_{2l^{2}}(2l(l-1)+l+1,n))+\zeta^{2}(B_{2l^{2}}(2,n)\\ & \ \ \ -B_{2l^{2}}(l(l-1)+2,n)-B_{2l^{2}}(l(l-1)+l+2,n)+B_{2l^{2}}(2l(l-1)+l+2,n))+......+\zeta^{l}(B_{2l^{2}}(l,n)\\ & \ \ \ +B_{2l^{2}}(l(l-1),n)-B_{2l^{2}}(l(l-1)+2l,n)-B_{2l^{2}}(2l(l-1)+l,n))+......+\zeta^{l(l-1)-1}(B_{2l^{2}}(l(l-1)-1,n)\\ & \ \ \ +B_{2l^{2}}(l(l-1)+l-1,n)-B_{2l^{2}}(2l^{2}-(l+1),n)+B_{2l^{2}}(2l(l-1)+2l-1,n))
\end{align*}

$=\sum_{i=0}^{l(l-1)-1}d_{i,n}\zeta^{i}$.
\\ Here
\begin{equation*}
d_{i,n}=B_{2l^{2}}(i,n)\mp B_{2l^{2}}(l(l-1)+j,n)-B_{2l^{2}}(l^{2}+k,n)\pm B_{2l^{2}}((2\phi(l^{2})+l)+j,n),
\end{equation*}
where
\begin{equation*}
0\leq j \leq l-1, \ j\equiv i \ (\textup{mod}\ 2l^{2}), \ k\equiv i \ (\textup{mod}\  2l^{2}), \  0\leq i,k \leq \phi(2l^{2})-1,
\end{equation*}
and sign's are '$l$' periodic.
\end{proof}
\begin{Theorem} 
Let $l\geq3$ be a prime and $q=p^{r}\equiv 1 \ (\emph{mod}\ 2l^{2})$. If $n$, $1\leq n \leq 2l^{2}-3$ is an odd integer,    
 then a congruence for $J_{2l^{2}}(1,n)$ over $\mathbb{F}_{q}$ is given by
\begin{equation*}
J_{2l^{2}}(1,n)\equiv \begin{cases}
\zeta_{l^{2}}^{-w}(-1+\sum_{i=3}^{l} c_{i,n}(\zeta_{l^{2}}-1)^{i})(\textup{mod}\ (1-\zeta_{l^{2}})^{l+1})),\ \mbox{if $n=l^{2}$}, \\ \\ 
-\zeta_{l^{2}}^{-w(dl+1)}(-1+\sum_{i=3}^{l} c_{i,n}(\zeta_{l^{2}}-1)^{i})(-1+\sum_{i=3}^{l} c_{i,n}(\zeta_{l^{2}}^{(-1-dl)/2}-1)^{i})\\
\pmod {(1-\zeta_{l^{2}})^{l+1})}  \ \mbox{if $d\neq l$ odd integer and $n=dl$},\\ \\
  \zeta_{l^{2}}^{-w(n+1)}(-1+\sum_{i=3}^{l} c_{i,n}(\zeta_{l^{2}}-1)^{i})(-1+\sum_{i=3}^{l} c_{i,n}(\zeta_{l^{2}}^{n}-1)^{i})\\ \times(-1+\sum_{i=3}^{l} c_{i,n}
  (\zeta_{l^{2}}^{(1-l^{2})/2}-1)^{i})\pmod {(1-\zeta_{l^{2}})^{l+1})},\, \mbox{if \, $gcd(n,2l^{2})=1$}.  
\end{cases}
\end{equation*}
Where $c_{i,n}$ are as described in the Theorem \ref{T1}.
\end{Theorem}
\begin{proof}
The proof of the theorem is immediate from the above-mentioned Lemma's.
\end{proof}
\noindent\textbf{Remark 2:} If $n$ is even, $2\leq n\leq 2l^{2}-2$. The congruences for Jacobi sums $J_{2l^{2}}(1,n)$ can be calculated using the relation 
$J_{2l^{2}}(1,n)=\chi(-1)J_{2l^{2}}(1,2l^{2}-n-1)$. Also if $d$ in the theorem is even then $J_{2l^2}(1,dl)=\chi(-1)J_{2l^2}(1, 2l^2-dl-1)$ and $2l^2-dl-1$
is odd. Thus the congruences for $J_{2l^2}(1,n)$ gets completely determined and hence that of all Jacobi sums of order $2l^2$. Further as $1-\zeta_{2l^2}$
is a divisor of $1-\zeta_{l^2}$, the congruences remains same modulo $(1-\zeta_{2l^2})^{l+1}$.  
\\ \\  \textbf{Future Scope:} The study of cyclotomic numbers of order $2l^{2}$ and prime ideal decomposition of Jacobi sums of order $2l^{2}$ are among the future scopes.\\ \\ 
\noindent\textbf{Acknowledgment:} The authors would like to thank   Central University of Jharkhand, Ranchi, Jharkhand, India for the support during preparation of this research article.

Department of Mathematics, Central University of Jharkhand, Ranchi, Jharkhand- 835205, India. \\
 Email Addresses: ahmed.helal@cuj.ac.in and jagmohan.t@gmail.com

\end{document}